\documentclass[12pt,thmsa, reqno]{amsart}

\usepackage{mathrsfs}
\usepackage{amsfonts}
\usepackage{amssymb}
\usepackage{amsmath}

\usepackage[dvips]{graphics}
\input{amssym.def}
\input{amssym.tex}

\def\gnk{G_{n,k}}
\def\cgnk{\A_{n,k}}
\def\cgnj{\A_{n,j}}

\def\Cal{\mathcal}
\def\A{{\Cal A}}

\def\gnk{G_{n,k}}

\def\bbr{{\Bbb R}}

\def\bbc{{\Bbb C}}

\def\rnn{\bbr^{n+1}}

\def\const{{\hbox{\rm const}}}

\def\vol{{\hbox{\rm vol}}}

\def\gnk{G_{n,k}}

\def\rn{\bbr^n}
\def\sn{S^{n-1}}

\def\part{\partial}
\def\intl{\int\limits}
\def\b{\beta}

\def\Gam{\Gamma}
\def\Om{\Omega}
\def\a{\alpha}
\def\om{\omega}

\def\vp{\varphi}

\def\gam{\gamma}
\def\Lam{\Lambda}
\def\sig{\sigma}
\def\lam{\lambda}
\def\z{\zeta}
\def\th{\theta}

\def\t{\tau}

\def\chi{{\bf 1}}

\newtheorem{theorem}{Theorem}[section]
\newtheorem{lemma}[theorem]{Lemma}

\theoremstyle{definition}

\newtheorem{example}[theorem]{Example}

\theoremstyle{remark}
\newtheorem{remark}[theorem]{Remark}

\theoremstyle{corollary}
\newtheorem{corollary}[theorem]{Corollary}
\newtheorem{proposition}[theorem]{Proposition}
\newtheorem{conjecture}[theorem]{Conjecture}

\numberwithin{equation}{section}

\newcommand{\be}{\begin{equation}}
\newcommand{\ee}{\end{equation}}

\newcommand{\bea}{\begin{eqnarray}}
\newcommand{\eea}{\end{eqnarray}}
\newcommand{\Bea}{\begin{eqnarray*}}
\newcommand{\Eea}{\end{eqnarray*}}

\def\sideremark#1{\ifvmode\leavevmode\fi\vadjust{\vbox to0pt{\vss
 \hbox to 0pt{\hskip\hsize\hskip1em
\vbox{\hsize2cm\tiny\raggedright\pretolerance10000
 \noindent #1\hfill}\hss}\vbox to8pt{\vfil}\vss}}}%

\begin{document}

\title[ Norm Estimates]
{Norm Estimates for $k$-Plane Transforms and  Geometric  Inequalities}

\author{B. Rubin}
\address{
Department of Mathematics, Louisiana State University, Baton Rouge,
LA, 70803 USA}

\email{borisr@math.lsu.edu}

\subjclass[2010]{Primary 44A12; Secondary 52A40}



\keywords{Radon  transforms, Grassmann manifolds, geometric inequalities.}

\begin{abstract}
The  article is devoted to remarkable interrelation between the norm estimates for  $k$-plane transforms in weighted and unweighted $L^p$ spaces  and geometric integral inequalities for  cross-sections of measurable sets in $\rn$. We also consider more general $j$-plane to $k$-plane transforms on affine Grassmannians and    their compact  modifications. The article contains a series of new integral-geometric inequalities with sharp constants, explicit equalities,  conjectures, and open problems.
\end{abstract}

\maketitle

\section{Introduction}

\setcounter{equation}{0}

Let $\cgnk$  be the affine Grassmannian bundle
 of all unoriented  $k$-dimensional planes  in
$\rn$, $0 < k <n$. We denote by   $\gnk$  the  Grassmann manifold
 of  $k$-dimensional  linear subspaces of $\rn$, i.e., $k$-planes passing through the origin.
 Each $k$-plane $\tau \in \cgnk$  will be parameterized by the pair
$( \xi, u)$, where $\xi \in \gnk$ and $ u \in \xi^\perp$, the
orthogonal complement of $\xi $ in $\rn$.
  The manifold  $\cgnk$ is endowed with the product measure $d\t=d_*\xi du$,
where $d_*\xi$ is the
 $O(n)$-invariant probability measure  on $\gnk$ and $du$ denotes the  Euclidean volume element on $\xi^\perp$.

 We will be dealing with standard spaces $L^p (\cgnk)$ and the corresponding  the weighted spaces
\be\label{eeee}  L^p_\nu(\cgnk)=\{ f: ||f||_{p,\nu}\equiv |||\t|^\nu f(\t)||_p<\infty \},\qquad 1\le p\le \infty,\ee
where $||\cdot ||_p$ is the usual norm in $L^p(\cgnk)$  and $|\tau|$ stands for  the  Euclidean distance from the origin to  $\tau$.

The $k$-plane Radon-John transform   of a sufficiently good function $f$ on $\bbr^n$  is a function $R_kf$ on $\cgnk$ defined by
 \be\label{uu7}
 (R_kf)(\t)\equiv (R_kf)(\xi,u) =\intl_{x\in\tau} f(x) \,d_\tau x,
 \ee
where $d_\tau x$ is the  Euclidean volume element on $\t$. A  more general $j$-plane to $k$-plane transform (or $(j,k)$-transform for short) takes a
function $f$ on $\cgnj$ to a function $R_{j,k}f$ on $\cgnk$, $0\le j<k<n$, by the formula
 \be\label{uu7j}
(R_{j,k}f)(\t)\!\equiv \!(R_{j,k}f)(\xi,u)\! = \!\!\intl_{\z\subset\tau}\!\! f(\z) \,d_\tau \z.
\ee
Here $d_\tau \z$  stands for the canonical  measure on the affine Grassmann manifold of all $j$-dimensional planes in $\t$.
  In the case $j=0$, $0$-planes are just points, so that  $\A_{n,0}=\rn$ and $R_{0,k}=R_{k}$.

A background information about $k$-plane transforms can be found, e.g., in the books by Helgason \cite{H11}, Markoe \cite{Mar}, and   Gonzalez  \cite {Go10}, containing many references on this subject; see also   Keinert \cite{Ke} and Rubin \cite{Ru04b, Ru15}.
These publications are mainly focused on inversion and range characterization problems. More general $(j,k)$-transforms, which  fall into the scope of the general Helgason's group-theoretic double fibration setting,   were apparently first mentioned by Strichartz \cite [p. 701]{Str81}. The papers by Graev  \cite{Gr},  Gonzalez and Kakehi  \cite {GK},   Rubin \cite{Ru04d}, Rubin and Wang \cite{RW16a, RW16b} deal  with inversion formulas for these transforms. See also Strichartz \cite{Str86} regarding $L^2$ harmonic analysis on Grassmannian bundles.

A tremendous activity related to $L^p$-$L^q$ norm estimates for Radon transforms was initiated by pioneering papers of Solmon
\cite{So76, So79}, Oberlin \cite{Ob},  Calder\'on \cite{Cal},   Oberlin and Stein \cite{OS},   Strichartz \cite {Str81}, Christ \cite{Chr84},  Drury  \cite{Dru83, Dru84, Dru86, Dru89, Dru89a}. Sharp constants and extremizers for Radon-like  transforms were studied in the recent papers by Christ \cite{Chr14},  Drouot \cite{Dro14, Dro15, Dro16}, and  Flock \cite{Fl16}: see also  Baernstein II and  Loss \cite {BL}, Bennett, Bez and Gutiérrez \cite {BBG}, Gressman \cite{Gre1, Gre2},    Oberlin  \cite{ObR},  Tao and Wright \cite{TW} for further developments and perspectives. Norm inequalities with sharp constants  for $k$-plane transforms and $(j,k)$-transforms   in weighted spaces (\ref{eeee}) were obtained by Rubin \cite {Ru14}.

Besides afore-mentioned purely analytic issues, our concern  is related to the variety  of publications in the area of integral  geometry and probability dealing with volumes of cross-sections of geometrical objects. This direction of investigation amounts to the pioneering works by Steiner, Minkowski, Blaschke, and subsequent publications of many  authors. Numerous related references can be found  in  the books by  Burago and  Zalgaller  \cite{BuZ}, Gardner \cite{Ga06}, Schneider \cite{Sc13},     Schneider and  Weil  \cite{SchW}; see also  the papers by Schneider \cite{Sc85},    Grinberg  \cite{Gr91},  Lutwak \cite{Lut93}, Gardner \cite{Ga07},    Bianchi, Gardner and Gronchi  \cite{BGG},       Chasapis,  Giannopoulos and Liakopoulos  \cite{CGL},      Dann,  Kim and   Yaskin  \cite{DKY},  Dann,  Paouris and  Pivovarov \cite{DPP},     Dafnis and Paouris \cite{DP}, to mention a few.

In fact, the functional-analytic and geometric approaches are closely related because many results in analysis  can be  converted into similar statements in geometry, and vice versa. For example, an intimate connection of the Busemann intersection inequality \cite{Bu}  to the Oberlin-Stein $L^p$-$L^q$ estimate for the Radon transform \cite {OS} was noticed by Lutwak  \cite [p. 162,  (3)]{Lut93}; see subsection \ref{unwe} for details.  One of our aims in the present article is to extend Lutwak's remarkable observation to more general Radon-like transforms, derive  new inequalities, and formulate some conjectures.

\vskip 0.3 truecm

\noindent {\bf Plan of the Paper.}  In Section 2 we review basic  facts related to norm inequalities for operators  (\ref{uu7})-(\ref{uu7j})  and consider the corresponding weighted and unweighted  inequalities for planar sections of  measurable sets in $\rn$. The main focus is made on sharp constants and their asymptotics as $n\to \infty$.
Section 3  contains description of the transition from  Radon transforms on affine Grassmannians to   Funk type transforms on the unit sphere and compact Grassmannians. The results of this section are used in Section 4 to obtain new geometric inequalities for  dual intrinsic volumes  and central sections of star sets.

\vskip 0.3 truecm

\noindent {\bf Main Results.} The main new results of the paper are the general inequalities (\ref {op1})  and (\ref {op1as}), Conjectures \ref{conj} and \ref{abji}, Theorems \ref{lonm3p0} and \ref{mn1}. The paper contains  many  consequences of these statements which might be of independent interest.

\vskip 0.3 truecm

\noindent { \bf Notation.}  We will be using the same notation as in (\ref{eeee})-(\ref{uu7j}); $V_n (\cdot)$ denotes the $n$-dimensional volume function;
 $ \; \sigma_{n-1}  = 2 \pi^{n/2}/
 \Gamma(n/2)$ is the area of the unit sphere $S^{n-1}$ in $\rn; \;$ $d\theta$ stands for the surface element of $S^{n-1}$;  $d_*\theta$  ($= d\theta/\sigma_{n-1}$) is  the corresponding normalized surface element. The notation  $d_* (\cdot )$ is also used for the probability measure on Grassmann manifolds. We write  $B_{n} = \{ x \in \bbr^n: \ |x| \le 1 \}$ for  the unit ball in $\rn$;
 $b_{n}$ denotes the  volume  of  $B_{n}$, so that
 \be\label{bn} b_{n} = \frac{\sigma_{n-1}}{n}=\frac{\pi^{n/2}}{  \Gamma (n/2+1)}.\ee
A compact subset of $\rn$ is called a {\it body} if it is the closure of its interior. As usual, for $1\le p\le \infty$, the notation $p'$ stands for the dual exponent, so that $\; 1/p+1/p'=1$.  We say that  an integral under consideration
 exists in the Lebesgue sense if it is finite when the corresponding integrand is replaced by its absolute value.

\section {Norm Estimates of $k$-Plane Transforms and Affine Sections of Measurable Sets}

In this section we review basic facts related to the action of $k$-plane transforms (\ref{uu7})  and  $(j,k)$-transforms (\ref{uu7j}) in weighted and unweighted $L^p$ spaces and derive some  integral-geometric inequalities for affine sections of measurable sets in $\rn$.  Our main concern is sharp constants in the corresponding inequalities.

\subsection {Weighted Norm Estimates} \label{lplp}
The following statements were proved in our paper \cite[Theorems 1.1 and 1.2]{Ru14}.

\begin{theorem} \label{kk5cds} Let $1\le p \le \infty$,  $\nu=\mu-k/p'$,  $\mu>k-n/p$,
\be \label{3458}
\om_{k,p,\mu}(n)=\pi^{k/2p'}\,\left (\frac{ \Gam(n/2)}{\Gam((n-k)/2)}\right )^{1/p}\, \frac{\Gam ((\mu +n/p -k)/2)}{\Gam ((\mu +n/p)/2)}.
\ee
Then $R_k$ is a linear bounded operator from $L^p_\mu(\bbr^n)$ to $L^p_\nu(\cgnk)$ with the norm
\be\label{norm}
||R_k|| = \om_{k,p,\mu}(n).\ee
\end{theorem}
\begin{theorem} \label{kk5cdj} Let $1\le p \le \infty$,  $\nu=\mu-(k-j)/p'$,  $\mu>k-n/p-j/p'$,
\be \label{3458j}
\om_{j,k,p,\mu} (n)\!=\!\pi^{(k-j)/2p'}\left (\frac{ \Gam((n\!-\!j)/2)}{\Gam((n\!-\!k)/2)}\right )^{1/p}
\frac{\Gam ((\mu \!+\!n/p\! -\!k\!+\!j/p')/2)}{\Gam ((\mu \!+\!n/p\!-\!j/p)/2)}.\ee
Then  $R_{j,k}$ is a linear bounded operator from $L^p_\mu(\cgnj)$ to $L^p_\nu(\cgnk)$ with the norm
\be\label{norm1}||R_{j,k}||=\om_{j,k,p,\mu} (n).\ee
\end{theorem}

Some comments are in order.

{\bf 1.}  The assumptions for $\mu$ and $\nu$ in these theorems are best possible.  Theorem \ref{kk5cds} is formally contained in  Theorem \ref{kk5cdj}, in particular, $\om_{k,p,\mu} (n)=\om_{0,k,p,\mu} (n)$.

\vskip 0.3 truecm

{\bf 2.} Theorems  \ref{kk5cds} and \ref{kk5cdj} were obtained in \cite{Ru14} as consequences of the corresponding norm inequalities for functions.  For example, to prove (\ref{norm}), we established that
\be \label{oopli} ||R_k f||_{L^p_\nu(\cgnk)}  \le  \om_{k,p,\mu}(n)\, ||f||_{L^p_\mu(\bbr^n)}\ee
 for {\it every} $f\in L^p_\mu(\bbr^n)$ and
\be \label{oopli1} ||R_k f_0||_{L^p_\nu(\cgnk)}  \ge  \om_{k,p,\mu}(n)\, ||f_0||_{L^p_\mu(\bbr^n)}\ee
 for {\it some  nonnegative} $f_0\in L^p_\mu(\bbr^n)$.

It means that (\ref{oopli}) is sharp in the class of nonnegative functions in $L^p_\mu(\bbr^n)$, which is smaller than the entire space $L^p_\mu(\bbr^n)$. This remark will guarantee the sharpness of geometric inequalities in  Subsection \ref{Weig}.

\vskip 0.3 truecm

{\bf 3.}
  In the case  $p=1$, we have exact equalities
 \be\label{eq}
 \intl_{\cgnk} (R_kf)(\t)\, |\t|^\mu\, d\t= \om_{k,1,\mu}(n) \intl_{\rn} f(x)\, |x|^\mu\, dx,\ee
 \be\label{eq1}
 \intl_{\cgnk} (R_{j,k} f)(\t)\, |\t|^\mu\, d\t= \om_{j,k,1,\mu}(n) \intl_{\cgnj} f(\z)\, |\z|^\mu\, d\z,\ee
 where $\mu>k-n$,
\be\label{eq2}
  \om_{k,1,\mu}(n)=\frac{ \Gam(n/2)}{\Gam((n-k)/2)}\, \frac{\Gam ((\mu +n -k)/2)}{\Gam ((\mu +n)/2)}, \ee
  \be\label{eq3}
  \om_{j,k,1,\mu}(n)= \frac{ \Gam((n-j)/2)}{\Gam((n-k)/2)}\, \frac{\Gam ((\mu +n -k)/2)}{\Gam ((\mu +n-j)/2)}.\ee
 These formulas can be found  in  our papers \cite[formula (2.17)]{Ru04b} and \cite[formula (2.20)]{Ru04d} in different notation. It is assumed that either side of (\ref{eq}) and  (\ref{eq1}) exists in the Lebesgue sense.
\vskip 0.3 truecm

{\bf 4.} We do not know if the exact equalities
\be\label{aascal}  ||R_k f||_{L^p_\nu(\cgnk)} = \om_{k,p,\mu}(n)\, ||f||_{L^p_\mu(\bbr^n)}, \ee
\be\label{bbscal}  ||R_{j,k} f||_{L^p_\nu(\cgnk)} = \om_{j,k,p,\mu}(n)\, ||f||_{L^p_\mu(\cgnj)},\ee
 are available for some $f$  when  $p \neq 1$. It may be a  challenging open problem.

\vskip 0.3 truecm

{\bf 5.}  One can easily find asymptotics of the norms (\ref{norm}) and (\ref{norm1}) as $n\to \infty$. Indeed, the well-known property of gamma functions \[\Gam (z+a)/\Gam (z+b) \sim z^{a-b},\qquad z\to \infty,\]
    (see, e.g., \cite{TE}) yields
  \be\label {as}
  \om_{j,k,p,\mu} (n)=n^{-(k-j)/2p'}\, (\om^0_{j,k,p} +o(1)), \qquad n\to \infty.\ee
 \be\label {as1}
 \om^0_{j,k,p}= (2\pi)^{(k-j)/2p'}\, p^{(k-j)/2}.\ee

\subsection{Unweighted $L^p$-$L^q$ Estimates} Let $0\le j< k<n$. It is known \cite[Corollary 2.6]{Ru04d} that for $f\in L^p (\cgnj)$, the integral $R_{j,k}f$ is finite provided $1\le p < (n-j)/(k-j)$ and this bound is sharp.
By (\ref{eq1}) with $\mu=0$, $R_{j,k}$ acts as a linear bounded operator from $L^1 (\cgnj)$ to $L^1 (\cgnk)$  with  the operator norm $1$.
Further, by the scaling argument  (cf. \cite[p. 118]{St}), if the estimate
\be\label{scal} || R_{j,k} f||_{ L^{q}(\cgnk)} \le c\, ||f||_{L^{p}(\cgnj)} \ee
holds for every $f\in L^p (\cgnj)$  with a constant $c$ independent of $f$, then, necessarily,
\be\label{scal2}
1/q=\frac{n-j}{n-k}\, (1/p)  -\frac{k-j}{n-k}.\ee
It means that the set of all pairs $(1/p, 1/q)$ which are  admissible in (\ref{scal}), is a subset of the half-open segment $(P,Q]$ in $[0,1]\times [0,1]$ connecting the points $P=((k-j)/(n-j), 0)$ and Q=(1,1). To achieve the desired boundedness result, the segment $(P,Q]$ must be replaced by the smaller closed segment  $[\tilde P,Q]$ with $\tilde P=((k+1)/(n+1), 1/(n+1))$, and we have the following theorem.

\begin{theorem}\label{OSBLCDF} Let $0<k<n$,
\[ p=\frac{n+1}{k+1}, \qquad q=n+1, \qquad \Om_k (n)=\left ( 2^{k-n}\, \frac{\sig_k^n}{\sig_n^k}\right )^{1/(n+1)}.\]
Then
\be\label{scamm} || R_k f||_{ L^{q}(\cgnk)} \le \Om_k (n)\, ||f||_{L^{p}(\rn)}.\ee
The equality sign in (\ref{scamm})  holds if and only if
\be\label{scamm1} f(x)=c\, (1+|Mx|^2)^{-(k+1)/2},\ee
where $c=\const$ and $M$ is an  invertible affine map.
\end{theorem}

This remarkable theorem combines results of several authors. The estimate (\ref{scamm}) with  unspecified numeric constant  is due to Oberlin and  Stein \cite{OS}, Christ \cite{Chr84}, and Drury \cite {Dru83, Dru84, Dru89}.  The question of optimal constant and extremizers  in (\ref{scamm}) was first considered by Baernstein and  Loss  \cite{BL} who  made several important conjectures regarding (\ref{scamm}) and related problems. A sharp constant  $\Om_k (n)$ was found by Drouot \cite{Dro14}. The equality case in (\ref{scamm}) for $k=n-1$ was studied by Christ \cite{Chr14}. The  choice of the extremizer  (\ref{scamm1}) for all $0<k<n$ was justified by  Drouot \cite{Dro14, Dro15} and Flock \cite{Fl16}.

The estimate  (\ref{scamm}) extends to all $1\le p\le (n+1)/(k+1)$ by interpolation and the bound $(n+1)/(k+1)$ is sharp.

To the best of our knowledge, a complete analogue of Theorem \ref{OSBLCDF} for $(j,k)$-transforms  is  unknown, and we formulate the following conjecture.

\begin{conjecture} \label {conj}
Let $\, 0\le j<k<n$,
\[p=\frac{n+1}{k+1}, \quad q=\frac{n+1}{j+1}, \quad \Om_{j,k} (n) =\left (\sig_j^{k-n}\, \sig_k^{n-j}\, \sig_n^{j-k}\,\right )^{1/(n+1)}.\]
Then
\be\label{scamm2} || R_{j,k} f||_{ L^{q}(\cgnk)} \le \Om_{j,k} (n)\, ||f||_{L^{p}(\cgnj)} \ee
where
\be\label{dryh}  \Om_{j,k} (n) =\left ( \sig_k^{n-j}\,\sig_j^{k-n}\, \sig_n^{j-k}\,\right )^{1/(n+1)}\ee
is the norm of the operator $R_{j,k}$.
The equality sign in (\ref{scamm2}) holds if and only if
\be\label{scamm3} f(\z)=c\, (1+|M\z|^2)^{-(k+1)/2},\qquad \z\in \cgnj,\ee
where $c=\const$  and $M$ is an  invertible affine map.
\end{conjecture}
Clearly, for  $j=0$ we have $\Om_{0,k} (n)= \Om_{k} (n)$, as in (\ref{scamm}).

The boundedness of $R_{j,k}$ from  $L^{(n+1)/(k+1)}(\cgnj)$ to $L^{(n+1)/(j+1)}(\cgnk)$ was stated by  Drury in \cite [Theorem 2]{Dru89a}, \cite [formula (12)]{Dru89}. Conjecture \ref{conj} is partly motivated by the following lemma, according to which
\be\label{dryhh}
|| R_{j,k} || \ge  \Om_{j,k} (n).\ee

\begin{lemma}
 If $f_0(\z)=(1+|\z|^2)^{-(k+1)/2}$, $\z\in \cgnj$, then
\be\label{dryh0} || R_{j,k} f_0||_{ L^{q}(\cgnk)} = \Om_{j,k} (n)\, ||f_0||_{L^{p}(\cgnj)}, \ee
where all parameters have the same meaning as in Conjecture \ref{conj}.
\end{lemma}
\begin{proof} By \cite [formula (2.12)]{Ru04d},
$ (R_{j,k} f_0)(\t)=\lam\, (1+|\t|^2)^{-(j+1)/2}$, where
\[\lam=\frac{\pi^{(k-j)/2}\, \Gam ((j+1)/2)}{\Gam ((k+1)/2)}=\frac{\sig_k}{ \sig_j}.\]
Note also that
\[ ||f_0||^{p}_{L^{p}(\cgnj)}=\intl_{\cgnj} (1+|\z|^2)^{-(n+1)/2}\, d\z=\sig_{n-j-1}  \intl_{0}^\infty
\frac{t^{n -j-1}\,dt}{(1+t^2)^{(n+1)/2}}\,  dt=\frac{\sig_n}{ \sig_j}\]
(use, e.g.,   \cite [formula 3.196 (2)]{GRy}). Similarly,
\[|| R_{j,k} f_0||^{q}_{ L^{q}(\cgnk)} =\lam^{q}\intl_{\cgnk} (1+|\t|^2)^{-(n+1)/2}\, d\z=\lam^{q}\, \frac{\sig_n}{ \sig_k}=\left (\frac{\sig_k}{ \sig_j}\right )^{q}\, \frac{\sig_n}{ \sig_k}.\]
This gives (\ref{dryh0}).
\end{proof}

\begin{remark} The following intriguing  inequality for the $k$-plane transform  due to Dann, Paouris and Pivovarov \cite[Theorem 1.3]{DPP} might be a nice addition to this subsection.

\begin{theorem}  If $f$ is a nonnegative bounded integrable function on $\rn$, $0<k<n$, then
\be\label {opsa} \intl_{\cgnk}  [(R_k f)(\t)]^{n+1}\, \frac {d\t} { || f|_\t ||^{n-k}_{\infty} } \le  \frac{b_k^{n+1}\, b_{n(k+1)}}{b_n^{k+1} \,  b_{k(n+1)} }\, ||f||_1^{k+1}, \ee
where $f|_\t$ is the restriction of $f$ onto $\t$.
\end{theorem}
\end{remark}

\subsection{Volume Estimates for Cross-Sections of Sets in $\rn$}

Inequalities for Radon transforms in the previous subsections give rise to a series of estimates
 for cross-sections of measurable subsets of $\rn$. The only restriction on these subsets is the finiteness of the  right-hand side in the corresponding estimates.  The estimates  contain several parameters that can be chosen depending on our need.

 \subsubsection {Weighted Estimates}  Given a measurable set $S$ in $\rn$, let $\chi_S (x)$ be the indicator function of $S$, that is, $\chi_S (x)\equiv 1$ if $x\in S$ and $\chi_S (x)\equiv 0$ if $x\notin S$. Then, for $\t \in \cgnk$, we have $V_k (S\cap \t)=(R_k\chi_S)(\tau)$, and  Theorem \ref{kk5cds} yields
 the following statement.

 \begin{proposition}   Let
 \[0< k< n, \quad 1\le p \le \infty,  \quad \nu=\mu-k/p', \quad   \quad \mu>k-n/p.\]
  Then
\be\label {op1}  \intl_{\cgnk} [V_k (S\cap \t)]^p\, |\t|^{\nu p}\, d\t \le \om^p_{k,p,\mu}(n) \,\intl_{S} |x|^{\mu p}\, dx,\ee
where $\om_{k,p,\mu}(n)$ is defined by (\ref{3458}), so that
\be\label {asn}
  \om_{k,p,\mu} (n)=n^{-k/2p'}\, \left( (2\pi)^{k/2p'}\, p^{k/2} +o(1) \right), \qquad n\to \infty.\ee
   \end {proposition}

   Note that the principal term of the asymptotics of $\om_{k,p,\mu} (n)$ is independent of $\mu$. The case $\mu=0$ gives the following corollary.
\begin{corollary}  If $1\le p<n/k$, then
\be\label {op3}
\intl_{\cgnk}  [V_k (S\cap \t)]^p\,  |\t|^{-k(p-1)}\, d\t \le \om^p_{k,p,0}(n) \, V_n (S), \ee
   \[\om_{k,p,0}(n)  = \pi^{k/2p'}\,\left (\frac{ \Gam(n/2)}{\Gam((n-k)/2)}\right )^{1/p}\, \frac{\Gam ((n/p -k)/2)}{\Gam (n/2p)}.\]
\end{corollary}
\begin{remark}
   If $p=1$, $\mu>k-n$, then (\ref{eq}) yields an explicit equality
      \be\label {op4}
   \intl_{\cgnk} V_k (S\cap \t)\, |\t|^{\mu}\, d\t = \om_{k,1,\mu}(n) \intl_{S} |x|^{\mu}\, dx,\ee
where $\om_{k,1,\mu}(n)$  is a constant (\ref{eq2}).  In particular, for $\mu=0$,
 \be\label {op5}
  \intl_{\cgnk} V_k (S\cap \t)\,  d\t =V_n (S),\ee
which is a well-known  consequence of  Fubini's theorem.
\end{remark}

  More generally,  assuming $\z\in \cgnj$, $\t\in \cgnk$, $0\le j<k<n$, and setting $f(\z)=(R_j \chi_S)(\z)= V_j (S\cap \z)$ , we obtain
\[(R_{j,k} f)(\t)=(R_{j,k} [ R_j \chi_S])(\t)=(R_k \chi_S)(\t)=V_k (S\cap \t).\]
Now, Theorem \ref{kk5cdj} gives the following inequality for mean volumes of cross-sections of different dimensions:
\be\label {op1as} \intl_{\cgnk}  [V_k (S\cap \t)]^p\, |\t|^{\nu p}\, d\t\le  \om^p_{j,k,p,\mu} (n)\, \intl_{\cgnj} [V_j (S\cap \z)]^{\mu p}\, d\z.\ee
 Here $\mu, \nu, p$ and $  \om_{j,k,p,\mu} (n)$  have the same meaning as in Theorem \ref{kk5cdj}.

 \subsubsection {Unweighted Estimates} The  inequality (\ref{scamm}) yields
\be\label {otra} \intl_{\cgnk}  [V_k (S\cap \t)]^{n+1}\, d\t\le  2^{k-n}\, \frac{\sig_k^n}{\sig_n^k} \, [V_n (S)]^{k+1}. \ee
A more general  inequality follows from (\ref{scamm2}) and has the form
\bea\label {op1ia} &&\Big  (\,\intl_{\cgnk}  [V_k (S\cap \t)]^{(n+1)/(j+1)}\, d\t\Big )^{(j+1)/(n+1)}  \nonumber\\
\label {op1ia} &&\le \Om_{j,k} (n)\, \Big  (\,\intl_{\cgnj} [V_j (S\cap \z)]^{(n+1)/(k+1)}\, d\z\Big )^{(k+1)/(n+1)}\eea
provided that Conjecture \ref{conj} is true.

\begin{remark} Apparently the weighted inequalities (\ref{op1}) and (\ref  {op3}) are not sharp because the subclass of  indicator functions is much smaller than that in Theorems  \ref{kk5cds}  and \ref{kk5cdj}. Regarding unweighted case, the following sharp result  is due to Gardner \cite [Theorem 7.8]{Ga07}. For the sake of convenience, we formulate it in our notation.

\begin{theorem}\label {KIOP} Let $S$ be a bounded Borel set in $\rn$, $1\le k\le n$. Then
\be\label {ogra} \intl_{\cgnk}  [V_k (S\cap \t)]^{n+1}\, d\t\le   \frac{b_k^{n+1}\, b_{n(k+1)}}{b_n^{k+1} \,  b_{k(n+1)} }\, [V_n (S)]^{k+1} \ee
with equality when $k>1$ if and only if $S$ is an $n$-dimensional ellipsoid, modulo a set of measure zero, and when $k=1$ if and only if $S$ is a convex body,  modulo a set of measure zero.
\end{theorem}

The estimate (\ref{ogra}) is sharper than (\ref{otra}) and agrees with (\ref{opsa}). A more general inequality
\be\label {ogra1} \intl_{\cgnk}  [V_k (S\cap \t)]^{m+1}\, d\t\le   \frac{b_k^{m+1}\, b_{n+km}}{b_n^{(n+km)/n} \,  b_{k+km} }\, [V_n (S)]^{1+km/n},  \ee
where $S$ is a convex body and $m \in \{1,\ldots, n\}$ is due to Schneider \cite{Sc85}.
Note that (\ref{ogra}) and (\ref{ogra1}) assume the set $S$ to be bounded, while (\ref{otra}) holds for arbitrary (not necessarily bounded) measurable set of finite measure.

\end{remark}

\section{Transition from $\rn$ and Affine Grassmannians to the Sphere and Compact Grassmannians}

Theorems  of the previous section can be converted into the similar statements for the Funk type transforms on the sphere and  Grassmann manifolds by making use of  the stereographic projection.  A remarkable  interrelation between diverse integral operators on $\rn$ and $S^n$  is known for many years and the corresponding transition formulas can be found in numerous publications; see, e.g.,   Mikhlin \cite[pp. 35-36]{Mi62},     Berenstein,   Casadio Tarabusi and  Kurusa \cite{BCK}, Drury \cite{Dru89a}, Rubin \cite [Section 5.2]{Ru15}, to mention a few. Below we recall the reasoning from our paper \cite{Ru04d} which has proved to be especially helpful in the general context of Grassmannians.

We consider the Euclidean space $\rn=\bbr e_1 \oplus \ldots
 \oplus\bbr e_n$ as a coordinate hyperplane in $\bbr^{n+1}=\bbr e_1 \oplus
\ldots \oplus\bbr e_{n+1}$. Given a linear subspace $V$ of $\rnn$
and a positive integer $k<\dim V$, we denote by $G_k (V)$ the
Grassmann manifold of all $k$-dimensional linear subspaces of $V$.
In particular,  $G_k(\rn)=\gnk, \; G_{k+1}(\rnn)=G_{n+1,
k+1}$. To each affine $k$-plane $\t$ in $\rn$ we associate a
$(k+1)$-dimensional linear subspace $\tau_0$ in $\rnn$ containing
 the ``lifted'' plane $\t +e_{n+1}$. This leads to a
map
\be \cgnk \ni \t=\xi +u \xrightarrow {\quad \gam_k \quad }
\tau_0 =\gam_k (\t)= \xi \oplus \bbr u_0 \in G_{n+1, k+1},\ee
where
 \be u_0=\frac{u+e_{n+1}}{|u+e_{n+1}|}=\frac{u+e_{n+1}}{\sqrt
{1+|u|^2}} \in S^n. \ee
If $\, \theta=d(\tau_0)$ is the geodesic
distance (on $S^n$) between the north pole $e_{n+1}$ and the
$k$-dimensional totally geodesic submanifold $S^n \cap \tau_0$, then
\be\label {lon} |\t|=|u|=\tan \theta.\ee
In a similar way, we define a map $\gam_j$ by
\be \cgnj \ni \z=\eta +v \xrightarrow {\quad \gam_j \quad }
\z_0 =\gam_j (\z)=\eta \oplus \bbr v_0\in G_{n+1, j+1},     \ee
 \[ v_0=\frac{v+e_{n+1}}{\sqrt
{1+|v|^2}} \in S^n,   \quad     |\z|=|v|=\tan \om,  \quad \om=d(\z_0), \]
and denote
\[ (\Lam_j f)(\z_0)=f (\gam_j^{-1}(\z_0)), \qquad f\equiv f(\z), \qquad \z\in \cgnj,\]
\[ (\Lam_k \vp)(\t_0)=\vp (\gam_k^{-1}(\t_0)),  \qquad \vp\equiv \vp(\t),  \qquad \t\in \cgnk.\]
The maps $\gam_j$ and $ \gam_k$ are one-to-one up to the corresponding subsets of measure zero.

For $0\le j<k<n$, consider the Funk type transform
  \be \label{gfunk} (F_{j+1,k+1}g)(\t_0)=\intl_{G_{j+1}(\t_0)} g (\z_0)\, d_{\t_0} \z_0, \qquad \t_0 \in G_{n+1,
  k+1}, \ee
  that integrates $g(\z_0)$ over the set of all
  $(j+1)$-dimensional linear subspaces of $\t_0$ against the canonical
  probability measure on $G_{j+1}(\t_0)$.  If $j=0$, (\ref{gfunk}) is the classical Funk transform of even functions on $S^n$ \cite{H11, Ru15}.

The following statement is a reformulation of Lemmas 3.2 and 3.4 from \cite{Ru04d}  adapted to our notation.

\begin{lemma}\label {mlio} Let  $0\le j<k<n$,
\be a\!=\!\frac{\sig_k}{\sig_j}, \quad \rho_1 (\z)\!=\!(1\!+\!|\z|^2)^{-(k +1)/2}, \quad \rho_2 (\t)\!=\!(1\!+\!|\t|^2)^{-(j
  +1)/2}.\ee
Then
\bea\label{34m}  R_{j,k}f&=&a\rho_2 \Lam_k^{-1}F_{j+1,k+1} \Lam_j \rho_1^{-1}f, \qquad  f: \cgnj  \to \bbc, \\
{}\nonumber\\
  \label{34mg}  F_{j+1,k+1}g&=&a^{-1}\Lam_k\rho_2^{-1} R_{j,k}  \rho_1\Lam_j^{-1}g, \qquad  g: G_{n+1, j+1}  \to \bbc, \eea

\be\label{34mg1}
  \intl_{\cgnk} \vp (\t)\, d\t=\frac{\sig_n}{\sig_k}\intl_{G_{n+1, k+1}} \frac{(\Lam_k \vp)(\t_0)}{(\cos \, d(\t_0))^{n+1}}\,  d_*\t_0, \ee

\be\label{34mg2}
  \intl_{G_{n+1, k+1}} g (\t_0)\, d_*\t_0=\frac{\sig_k}{\sig_n}\intl_{\cgnk} \frac{(\Lam_k^{-1} g)(\t)}{(1+|\t|^2)^{(n+1)/2}}\,  d_\t. \ee

\end{lemma}

\subsection{Unweighted Estimates}  Lemma \ref{mlio} implies the following statement.

\begin{corollary}\label {mlio2} Let

\be\label {lonm} p=\frac{n+1}{k+1}, \quad q=\frac{n+1}{j+1}, \quad b=\left (\frac{\sig_j}{\sig_n}\right )^{1/p}, \quad c=\left (\frac{\sig_k}{\sig_n}\right )^{1/q}.\ee
Then

\be\label {lonm1}
||\Lam_j \rho_1^{-1}f ||_{L^p(G_{n+1, j+1})}=b\, ||f||_{L^p(\cgnj)}, \ee

\be\label {lonm2}
||\Lam_k\rho_2^{-1} R_{j,k}f ||_{L^q(G_{n+1, k+1})}=c \,||R_{j,k}f||_{L^q(\cgnk)}. \ee

\end{corollary}
\begin{proof} Denote by $I$ and $J$  the left-hand sides of (\ref {lonm1}) and (\ref {lonm2}), respectively.   By (\ref{34mg2}) (with $k$ replaced by $j$),
\[I^p=\frac{\sig_j}{\sig_n}\intl_{\cgnj} \frac{\Lam_j^{-1} [(\Lam_j \rho_1^{-1}f)^p ](\z)}{(1+|\z|^2)^{(n+1)/2}}\,  d\z=\frac{\sig_j}{\sig_n}\, ||f||^p_{L^p(\cgnj)}.\]
 Similarly,
 \[
J^q=\frac{\sig_k}{\sig_n}\intl_{\cgnk} \frac{\Lam_k^{-1} [(\Lam_k \rho_2^{-1} R_{j,k}f)^q](\t)}{(1+|\t|^2)^{(n+1)/2}}\,  d\t=\frac{\sig_k}{\sig_n}\, ||R_{j,k}f||^q_{L^q(\cgnk)}.\]
 \end{proof}

\begin{theorem}\label {lonm3p0} Let  $0\le j< k<n$,
\[ p=\frac{n+1}{k+1}, \qquad q=\frac{n+1}{j+1}, \qquad \Om_{j,k} (n) =\left ( \sig_k^{n-j}\,\sig_j^{k-n}\, \sig_n^{j-k}\,\right )^{1/(n+1)}.\]

\noindent {\rm (i)} If $ F_{j+1,k+1}$ is bounded from $L^p(G_{n+1, j+1})$ to $L^q(G_{n+1, k+1})$, then $R_{j,k}$ is bounded from $L^p(\cgnj)$ to $L^q(\cgnk)$ and
\be\label {lonm3}
\Om_{j,k} (n) \le ||R_{j,k}||\le \Om_{j,k} (n) || F_{j+1,k+1} ||.\ee

\noindent {\rm (ii)}  Conversely, if $R_{j,k}$ is bounded from $L^p(\cgnj)$ to $L^q(\cgnk)$, then    $ F_{j+1,k+1}$ is bounded from $L^p(G_{n+1, j+1})$ to $L^q(G_{n+1, k+1})$ and
\be\label {lonm4}
1 \le  || F_{j+1,k+1} || \le \Om^{-1}_{j,k} (n)\, ||R_{j,k}||.\ee
\end{theorem}
\begin{proof} {\rm (i)}  Using successively (\ref{lonm2}), (\ref{34m}) and  (\ref{lonm1}), we obtain
\bea
||R_{j,k} f||_q  &=&\frac{1}{c}\, ||\Lam_k \rho_2^{-1}R_{j,k}f||_q = \frac{a}{c}\,  ||F_{j+1,k+1}\Lam_j \rho_1^{-1} f||_q \nonumber\\
&\le& \frac{a}{c}\, ||F_{j+1,k+1}|| \, ||\Lam_j \rho_1^{-1} f||_p = \frac{ab}{c}\, ||F_{j+1,k+1}|| \, || f||_p. \nonumber\eea
Hence, $||R_{j,k}||\le (ab/c )\, ||F_{j+1,k+1}||=\Om_{j,k} (n) || F_{j+1,k+1}||$. The left inequality in (\ref {lonm3}) mimics (\ref{dryhh}).

 {\rm (ii)} The successive use of (\ref{34mg}), (\ref{lonm2}) and (\ref{lonm1}) yields
\bea
||F_{j+1,k+1}g||_q &=&\frac{1}{a}\, ||\Lam_k \rho_2^{-1}R_{j,k} \rho_1\Lam_j^{-1} g||_q =\frac{c}{a}\,||R_{j,k} \rho_1 \Lam_j^{-1} g||_q\nonumber\\
&\le& \frac{c}{a}\,||R_{j,k} ||\, ||\rho_1 \Lam_j^{-1} g||_p =\frac{c}{ab}\,||R_{j,k} ||\, ||\Lam_j \rho_1^{-1}\rho_1 \Lam_j^{-1} g||_p\nonumber\\
&=&\frac{c}{ab}\,||R_{j,k} ||\, || g||_p=\Om^{-1}_{j,k} (n)\,||R_{j,k} ||\,|| g||_p.\nonumber\eea
Hence $||F_{j+1,k+1}||\le \Om^{-1}_{j,k} (n)\,||R_{j,k} ||$. The left inequality in (\ref {lonm4}) follows from the obvious equality $||F_{j+1,k+1}\chi ||_q =||\chi ||_p =1$.
 \end{proof}

The inequalities (\ref {lonm3}) and (\ref {lonm4}) imply the following statement.
\begin{corollary}\label{corcon} Suppose that  Conjecture \ref{conj} is true. Then, for all $0\le j< k<n$,
\be\label {lonm5}
||R_{j,k}||=\Om_{j,k} (n) || F_{j+1,k+1} ||\ee
in the  corresponding $L^p$-$L^q$ setting, as in Theorem \ref{lonm3p0}. In particular, if $|| F_{j+1,k+1} ||=1$, then $||R_{j,k}||=\Om_{j,k} (n)$.
\end{corollary}

In view of forthcoming applications, for the sake of convenience we
replace $n$ by $n-1$, $k$ by $k-1$, and $j$ by $j-1$. Then for all $0\le j<k<n$, Corollary \ref{corcon} implies

\be\label {alonm6}
\Big ( \intl_{G_{n, k}} |(F_{j,k} \vp)(\t_0)|^{n/j}\, d_*\t_0 \Big )^{j/n} \le \Big ( \intl_{G_{n, j}} |\vp (\z_0)|^{n/k}\, d_*\z_0 \Big )^{k/n}.\ee
If Conjecture \ref{conj} is true, this inequality is sharp.
A similar inequality without sharp constant was outlined by Drury   \cite[formula (10)]{Dru89}, \cite[Theorem 1]{Dru89a}.

\begin{remark} \label {xalonm6} The case $j=0$ deserves particular mentioning. By Theorem \ref{OSBLCDF}, $||R_{k}||=\Om_{k} (n)$. Hence  (\ref{lonm5}) yields
\[
 || F_{1,k} ||_{L^p(G_{n, 1})\to L^q(G_{n, k})}=1, \quad  p=n/k, \quad q=n, \quad 0< k<n.\]
  Identifying  functions on $G_{n, 1}$ with  even functions on $\sn$ and noting that
\be\label{funk} (F_{1,k} \vp)(\t_0) \equiv (F_k\vp)(\t_0)=\intl_{\sn \cap \t_0} \vp (\th)\, d_{\t_0}\th\ee
  is  the  Funk type transform that integrates $\vp$ over $(k-1)$-dimensional geodesics $S^{n-1} \cap \t_0$ with respect to the corresponding probability measure  \cite[p. 133]{H11},  we  obtain a sharp inequality
\be\label {lonm6}
\Big ( \intl_{G_{n, k}} |(F_k \vp)(\t_0)|^n\, d_*\t_0 \Big )^{1/n} \le \Big ( \intl_{\sn} |\vp (\th)|^{n/k}\, d_*\th \Big )^{k/n}.\ee
Because odd functions are annihilated by $F_k$, the assumption of evenness of $\vp$  can be dropped. Indeed, setting
$\vp=\vp_+ + \vp_- $, where $\vp_{\pm} (\th)=[\vp (\th) \pm \vp (-\th)]/2$, we have
\bea
&&\Big ( \intl_{G_{n, k}} |(F_k \vp)(\t_0)|^n\, d_*\t_0 \Big )^{1/n}= \Big ( \intl_{G_{n, k}} |(F_k \vp_+)(\t_0)|^n\, d_*\t_0 \Big )^{1/n}\nonumber\\
&&\le \Big ( \intl_{\sn} |\vp_+ (\th)|^{n/k}\, d_*\th \Big )^{k/n}=\frac{1}{2} \Big ( \intl_{\sn} |[\vp (\th) + \vp (-\th)]|^{n/k}\, d_*\th \Big )^{k/n}\nonumber\\
&&\le \frac{1}{2} \, \Big [\Big ( \intl_{\sn} |\vp (\th)|^{n/k}\, d_*\th \Big )^{k/n}  +   \Big ( \intl_{\sn} | \vp (-\th)|^{n/k}\, d_*\th \Big )^{k/n}\Big ]\nonumber\\
&&=\Big ( \intl_{\sn} |\vp (\th)|^{n/k}\, d_*\th \Big )^{k/n}.\nonumber\eea

The inequality (\ref{lonm6}) provides additional information to the corresponding results of Christ \cite[Theorem 2.1 (B)]{Chr84} and Drury \cite[Theorem 1]{Dru89}, where the  sharp constant is not specified. An alternative derivation of (\ref{lonm6}) can be found in Drouot \cite[Theorem 2]{Dro16}.
\end{remark}

\subsection{Weighted Estimates}
We combine Lemma \ref{mlio} with Theorem \ref{kk5cdj}  and replace $n+1$ by $n$, $k+1$ by $k$, and $j+1$ by $j$, respectively. This gives the following statement.

\begin{theorem}\label {mn1} Let $0\le j< k<n$, $\z_0 \in G_{n,j}$, $\t_0 \in G_{n,k}$,
\[\a (\t_0)= (\sin d(\t_0))^{\nu p}\, (\cos \, d(\t_0))^{(j-\nu) p-n}, \] \[\b(\z_0)=  (\sin d(\z_0))^{\mu p} \,(\cos \, d(\z_0))^{(k-\mu) p-n}.\]
Suppose that \[1\le p \le \infty,  \qquad \nu=\mu-(k-j)/p', \qquad \mu>k-n/p-j/p'.\]
 Then
\be\label{0kiy} \Big (\intl_{G_{n,k}}\,|(F_{j,k} \vp)(\t_0)|^p\, \a (\t_0)\, d_*\t_0 \Big )^{1/p}\le c\, \Big (\intl_{G_{n,j}} |\vp(\z_0)|^p\,\b(\z_0)\, d_*\z_0 \Big )^{1/p},\ee
where $c$ is a sharp constant having the form
\be\label{0kiyy} c=  \left (\frac{ \Gam(k/2)}{\Gam(j/2)}\right )^{1/p'}\, \left (\frac{ \Gam((n-j)/2)}{\Gam((n-k)/2)}\right )^{1/p}\, \frac{\Gam ((\mu +n/p -k+j/p')/2)}{\Gam ((\mu +n/p-j/p)/2)}.\ee
If $p=1$, (\ref{0kiy}) becomes an explicit equality.
\end{theorem}
\begin{proof} Let
\[ A= \intl_{\A_{n,k}} |(R_{j,k} f)(\t)|^p \, |\t|^{\nu p}\, d\t, \qquad B=\intl_{\A_{n,j}} |f(\z)|^p \, |\z|^{\mu p}\, d\z.\]
By Theorem \ref{kk5cdj},
\be\label{kiy99}
A^{1/p}\le \om_{j,k,p,\mu} (n)\, B^{1/p}.\ee
Our aim is to  covert $A$ and $B$ into the corresponding integrals over compact Grassmannians with $R_{j,k} f$ represented by the relevant  Funk type transform. We make use of
 (\ref{34mg1}) with $k$ replaced by $j$ and $\t$ by $\z$. Setting
\[ \vp(\z)=|f(\z)|^p \, |\z|^{\mu p}, \qquad \tilde h  (\z_0)\equiv h(\gam^{-1} (\z_0))=|\tilde f(\z_0)|^p \;(\tan d(\z_0))^{\mu p},\]
  we obtain
\be\label{lou} B=\frac{\sig_{n}}{\sig_{j}}\intl_{G_{n+1, j+1}} \frac{|\tilde f(\z_0)|^p \;(\tan d(\z_0))^{\mu p}}{(\cos \,d(\z_0))^{n+1}}\, d_*\z_0. \ee
To transform $A$, we observe that  by (\ref{34m}),
\be\label{lou1}
R_{j,k} f= \frac{\sig_k}{\sig_j} \, \rho_2 \Lam_k^{-1}F_{j+1,k+1} \Lam_j \rho_1^{-1}f. \ee
 Hence, by (\ref{34m}) and (\ref{34mg1}),
 \bea
A&=&\left (\frac{\sig_k}{\sig_j}\right )^p \,\intl_{\A_{n,k}} |\rho_2 (\t)\, (\Lam_k^{-1}F_{j+1,k+1} \Lam_j \rho_1^{-1}f)(\t)|^p\, |\t|^{\nu p}\, d\t\nonumber\\
&=&\left (\frac{\sig_k}{\sig_j}\right )^p \,\frac{\sig_{n}}{\sig_{k}}\intl_{G_{n+1, k+1}} |(F_{j+1,k+1}  \Lam_j \rho_1^{-1}f)(\t_0)|^p\nonumber\\
 &\times&  \frac{(\tan d(\t_0))^{\nu p}\, (\cos \,d(\t_0))^{(j+1)p}}{(\cos \,d(\t_0))^{n+1}}\, d_*\t_0\nonumber\eea
because
\[(\Lam_k \rho_2)(\t_0)= (1+(\tan d(\t_0))^2 )^{-(j+1)/2}=(\cos \,d(\t_0))^{j+1}.\]
Now we change the notation by setting $\vp (\z_0)= (\Lam_j \rho_1^{-1}f)(\z_0)$. Then obvious simplification allows us to write (\ref{kiy99}) in the form
\bea\label{kiy7} &&\Big (\intl_{G_{n+1,k+1}} (\sin d(\t_0))^{\nu p}\, (\cos \, d(\t_0))^{(j+1-\nu) p-n-1} \,|(F_{j,k} \vp)(\t_0)|^p\, d_*\t_0 \Big )^{1/p}\nonumber\\
&&\le c \, \Big (\intl_{G_{n+1,j+1}} (\sin d(\z_0))^{\mu p} \,(\cos \, d(\z_0))^{(k+1-\mu) p-n-1}\, |\vp(\z_0)|^p\, d_*\z_0 \Big )^{1/p},\nonumber\eea
 \[c=\om_{j,k,p,\mu} (n)\,\left (\sig_j /\sig_k\right )^{1/p'}.\]
 To complete the proof, it remains to replace $n+1$ by $n$, $k+1$ by $k$, and $j+1$ by $j$.    The equality sign in (\ref{0kiy}), when $p=1$,  holds by comment  {\bf 3} in Subsection \ref {lplp}.
 \end{proof}

 Let us set $\vp=F_j \psi$ in (\ref{0kiy}), where $\psi$ is a function on $S^{n-1}$. Using the equality $F_{j,k} F_j \psi=F_{k} \psi$,  we obtain the following result.

 \begin{corollary} Let $j, k, n, \a, \b$ and $c$ be the same as in Theorem \ref{mn1}. Then
 \be\label{00kiy} \Big (\intl_{G_{n,k}}\,|(F_{k} \psi)(\t_0)|^p\, \a (\t_0)\, d_*\t_0 \Big )^{1/p}\le c\, \Big (\intl_{G_{n,j}} |(F_{j}\psi)(\z_0)|^p\,\b(\z_0)\, d_*\z_0 \Big )^{1/p}\ee
 provided that the integral on the right-hand side exists in the Lebesgue sense.
\end{corollary}

For further purposes, we  formulate a particular case of Theorem \ref{mn1} corresponding to $j=1$, when a function $\vp$  on $G_{n, 1}$ can be identified with an even function on $\sn$ and the evenness restriction can be dropped, as we did in (\ref{lonm6}).

\begin{theorem}\label {mn21} Let $0< k<n$,  $\t_0 \in G_{n,k}$, $\theta=(\theta_1, \ldots, \theta_n)\in  S^{n-1}$. We set
\[\a_1 (\t_0)= (\sin d(\t_0))^{\nu p}\, (\cos \, d(\t_0))^{(1-\nu) p-n}, \]
\[\b_1(\theta)=  (1-\theta_n^2)^{\mu p/2} \,|\theta_n|^{(k-\mu) p-n},\]
and suppose that \[1\le p \le \infty,  \qquad \nu=\mu-(k-1)/p', \qquad \mu>k-n/p-1/p'.\]
 Then for every measurable function $\vp$ on $\sn$,
\be\label{kiy} \Big (\intl_{G_{n,k}}\,|(F_{k} \vp)(\t_0)|^p\, \a_1 (\t_0)\, d_*\t_0 \Big )^{1/p}\le c_1\, \Big (\intl_{S^{n-1}} |\vp(\theta)|^p\,\b_1 (\theta)\, d_*\theta \Big )^{1/p},\ee
where $c_1$ is a sharp constant having the form
\be\label{kiyy} c_1=  \left (\frac{ \Gam(k/2)}{\Gam(1/2)}\right )^{1/p'}\, \left (\frac{ \Gam((n-1)/2)}{\Gam((n-k)/2)}\right )^{1/p}\, \frac{\Gam ((\mu +n/p -k+1/p')/2)}{\Gam ((\mu +n/p-1/p)/2)}.\ee
If $p=1$, (\ref{kiy}) becomes an explicit equality. Specifically,
\be\label{pkiyva} \intl_{G_{n,k}}\, (F_{k} \vp)(\t_0)\, \tilde \a_1 (\t_0)\, d_*\t_0 = \tilde c_1\,\intl_{S^{n-1}} \vp(\theta)\,\tilde \b_1(\theta)\, d_*\theta,\ee
where
\[ \tilde \a_1 (\t_0)=(\sin d(\t_0))^{\mu}\, (\cos \, d(\t_0))^{1-\mu-n}, \qquad
\tilde \b_1(\theta)=  (1-\theta_n^2)^{\mu/2} \,|\theta_n|^{k-\mu-n},\]
\[\tilde c_1=  \frac{ \Gam((n-1)/2)}{\Gam((n-k)/2)}\, \frac{\Gam ((\mu +n -k)/2)}{\Gam ((\mu +n-1)/2)}, \qquad \mu >k-n.\]
\end{theorem}

Choosing $\mu, \nu, p, k$ and $n$ in Theorem \ref{mn21} in a suitable way, one can obtain  a series of new inequalities with sharp constants.

\begin{example} Let $\nu =0$, $p=n$. Then (\ref{kiy}) yields
\be\label{1kiy} \Big (\intl_{G_{n,k}}\,|(F_{k} \vp)(\t_0)|^n\, d_*\t_0 \Big )^{1/n}\le c_{0,1}\, \Big (\intl_{S^{n-1}} |\vp(\theta)|^n\,\b_0(\theta)\, d_*\theta \Big )^{1/n},\ee
where
\[\b_0(\theta)=  (1-\theta_n^2)^{(k-1)(n-1)/2} \,|\theta_n|^{k-1},\]
\be\label{kiyy} c_{0,1}=  \left (\frac{ \Gam(k/2)}{\Gam(1/2)}\right )^{1-1/n}\, \left (\frac{ \Gam((n-1)/2)}{\Gam((n-k)/2)}\right )^{1/n}\, \frac{\Gam ((1-k/n)/2)}{\Gam ((k-k/n)/2)}.\ee
\end{example}

Another  result  can be obtained if we apply  Theorem \ref{mn1} to  functions of the form $\vp= F_j\psi$ which are   Funk  transforms  over $(j-1)$-dimensional geodesics in $S^{n-1}$. Noting  that
$F_{j,k} \vp =F_{j,k} F_j\psi = F_k\psi$, we obtain the following inequality connecting Funk  transforms over geodesics of different dimensions.

\begin{theorem}\label {mn1a} Suppose that $j,k, \mu, \nu, p, \a, \b$ and $c$ have the same meaning as in Theorem \ref{mn1}. Then
\[ \Big (\intl_{G_{n,k}}\,|(F_{k} \psi)(\t_0)|^p\, \a (\t_0)\, d_*\t_0 \Big )^{1/p}\le c\, \Big (\intl_{G_{n,j}} |(F_{j} \psi)(\z_0)|^p\,\b(\z_0)\, d_*\z_0 \Big )^{1/p}.\]
\end{theorem}

\section{Cross-Sections of  Star Sets}

\subsection{Preliminaries}

 Theorems of the previous section imply a host of geometric inequalities and equalities.
 Below we give some examples. But  first we need to establish  terminology and recall some known facts.

A subset $L$ of $\rn$ is called a {\it star  set}  (with respect to the origin) if $\lam x \in L$ for every $x\in L$ and every $\lam \in [0, 1]$. A star set $L$  is uniquely determined by its radial function
\[
\rho_L (\th) = \sup \{c \ge 0: \, c \,\th \in L\}, \qquad \th \in \sn.
\]
 Everywhere in the following, we assume  the   set $L$ to be good enough so that the Lebesgue integrals on the right-hand  side of our formulas are finite. It means that $\rho_L$ belongs to the corresponding Lebesgue space on the sphere. For example, Klain \cite[Definition 2.3]{Kl} considered the so-called {\it $L^p$-stars} for which $\rho_L \in L^p(\sn)$.

For every star set $L\subset\rn$, the volume (i.e., the Lebesgue measure) of $L$ can be expressed in polar coordinates as
\be\label{kiya1}
V_n (L)=\frac{1}{n} \intl_{\sn} \rho_L^n (\th)\, d\th=b_{n} \intl_{\sn} \rho_L^n (\th)\, d_*\th. \ee
Similarly, for $0< k<n \, $ and   $\t_0 \in G_{n,k}$, the volume of the central cross-section $L\cap \t_0$ is
\be\label{kiya2}V_k (L\cap \t_0)=b_{k} \intl_{\sn \cap \t_0} \rho_L^k (\th)\, d_{\t_0}\th=b_{k} (F_{k} \rho_L^k)(\t_0),\ee
where $F_k$ is  the  Funk transform (\ref{funk}). The corresponding $p$-means
\be\label {op1ajzz}  \Big (\, \intl_{G_{n,k}} [V_k (L\cap \t_0)]^{p}\,  d_*\t_0\Big )^{1/p}\ee
were  introduced by Lutwak \cite{Lut79,  Lut88} and have proved to be useful in various geometrical considerations; see, e.g., \cite{DP},  \cite[Section 9.4 and Note 9.7 on p. 384]{Ga06}. Natural generalizations of  (\ref{kiya1}) and (\ref{kiya2}) are
{\it $m$th dual elementary mixed  volumes} (or  $m$th dual Quermassintegrals)
\be\label{dkiya1}
\tilde V_m (L)=b_{n} \intl_{\sn} \rho_L^m (\th)\, d_*\th=\frac{m}{n} \intl_L |x|^{m-n}\, dx, \ee
\be\label{dkiya2p}\tilde V_m (L\cap \t_0)=b_{k} \intl_{\sn \cap \t_0} \rho_L^m (\th)\, d_{\t_0}\th=b_{k} (F_{k} \rho_L^m)(\t_0), \quad \t_0 \in G_{n,k},\ee
which were introduced by Lutwak \cite{Lut75a, Lut} and studied by many authors; see, e.g.,
\cite[p. 158]{BuZ},  \cite[p. 409]{Ga06} and references therein. In particular, (\ref{dkiya1})-(\ref{dkiya2p}) naturally arise in the study
   of the Busemann-Petty type  comparison problems for convex bodies \cite{Had, RZ}.
The quantity (\ref{dkiya1}) can also be treated as the $m$-homogeneous rotation invariant valuation \cite {Al, Kl}.
Clearly,
\be\label{dkiya26} \tilde V_n (L) =V_n (L); \qquad \tilde V_k (L\cap \t_0)=V_k (L\cap \t_0). \ee
By the well-known property of the Funk transform \cite{Ru02b},
 \be\label{dokiya267}
 \intl_{G_{n,k}}  (F_k f)(\t_0)\, d_*\t_0=\intl_{\sn} f(\th)\, d_*\th,\ee
(\ref {dkiya1}) and (\ref{kiya2}) yield
\be\label{dkiya267}
 \tilde V_k (L) =\frac{b_{n}}{b_{k}}\intl_{G_{n,k}}  V_k (L\cap \t_0)\, d_*\t_0
 \ee
(set $f=\rho_L^k$ in (\ref{dokiya267}) and make use of (\ref{dkiya1}) and (\ref{dkiya2p})).

\subsection{Weighted Estimates}\label {Weig} We start with the  following
 \begin{remark}
All inequalities in this subsection hold with sharp constants. The sharpness is guaranteed by comment {\bf 1} in Subsection \ref{lplp} because every nonnegative function on the sphere can be regarded as a radial function of some star set. One should note that  if we impose additional restrictions on the class of star sets,  rather than  finiteness of the corresponding integrals,  the sharpness of the constants becomes unknown. For example, we do not know if our constants are sharp in the class of star bodies, when  $\rho_L$ is continuous.
\end{remark}

If $L$ is a star set, we can apply  Theorem \ref{mn21} to functionals  (\ref{kiya1})-(\ref{dkiya2p}).   For example, setting $\vp=\rho_L^m$ in (\ref{kiy}) and using (\ref{dkiya2p}), we obtain
\be\label{tmkiyv} \intl_{G_{n,k}}\, [\tilde V_m (L\cap \t_0)]^p\, \a_1 (\t_0)\, d_*\t_0 \le (c_1\,b_k)^p \,\intl_{S^{n-1}} \rho_L^{mp}(\theta)\,\b_1(\theta)\, d_*\theta,\ee
where $\a_1, \, \b_1, \, k,\, p, \,\t_0$ and $ c_1$ are the same as in (\ref{kiy}). If $p=1$, (\ref{tmkiyv}) becomes an explicit equality
\be\label{mkiyva} \intl_{G_{n,k}}\, \tilde V_m (L\cap \t_0)\, \tilde\a_1 (\t_0)\, d_*\t_0 = \tilde c_1\,b_k \intl_{S^{n-1}} \rho_L^{m}(\theta)\,\tilde\b_1(\theta)\, d_*\theta,\ee
where
$ \tilde\a_1 $, $\tilde\b_1$ and $\tilde c_1$ have the same meaning as in  (\ref{pkiyva}). For the sake of convenience, we recall
\[ \tilde \a_1 (\t_0)=(\sin d(\t_0))^{\mu}\, (\cos \, d(\t_0))^{1-\mu-n}, \qquad
\tilde \b_1(\theta)=  (1-\theta_n^2)^{\mu/2} \,|\theta_n|^{k-\mu-n},\]
\[\tilde c_1=  \frac{ \Gam((n-1)/2)}{\Gam((n-k)/2)}\, \frac{\Gam ((\mu +n -k)/2)}{\Gam ((\mu +n-1)/2)}, \qquad \mu >k-n.\]
If $\mu=0$,  (\ref{mkiyva}) becomes
\be\label{mkiyvab} \intl_{G_{n,k}}\,  \tilde V_m (L\cap \t_0)\,\frac { d_*\t_0}{(\cos \, d(\t_0))^{n-1}} = b_k\,\intl_{S^{n-1}} \rho_L^{m}(\theta)\,\frac { d_*\theta}{|\theta_n|^{n-k}}.\ee
In particular, if $m=k$, (\ref{dkiya26}) yields
\be\label{mkiyvabc} \intl_{G_{n,k}}\,   V_k (L\cap \t_0)\,\frac { d_*\t_0}{(\cos \, d(\t_0))^{n-1}} = b_k\,\intl_{S^{n-1}} \rho_L^{k}(\theta)\,\frac { d_*\theta}{|\theta_n|^{n-k}}.\ee

Further, choosing $p=n/k$ and $\mu=0$ in (\ref{tmkiyv}),  we obtain
\be\label{mkiyv3}  \intl_{G_{n,k}}\, [\tilde V_m (L\cap \t_0)]^{n/k}\, \a_2 (\t_0)\, d_*\t_0  \le \frac {(c_2\,b_k)^{n/k}}{ b_n} \,
\tilde V_{mn/k} (L),\ee
where
\[\a_2 (\t_0)=(\sin d(\t_0))^{(k-1)(k-n)/k}\, (\cos \, d(\t_0))^{1-k},\]
\[ c_2=\left (\frac{ \Gam(k/2)}{\Gam(1/2)}\right )^{1-k/n}\, \,
\left ( \frac{\Gam((n-1)/2)}{ \sig_{n-1}\,\Gam((n-k)/2)}\right )^{k/n}\, \frac{\Gam((1-k/n)/2)}{k\,\Gam((k-k/n)/2)}.\]

 If  $m=k$, (\ref{tmkiyv}) becomes
\be\label{mkiyv} \intl_{G_{n,k}}\, [V_k (L\cap \t_0)]^p\, \a_1 (\t_0)\, d_*\t_0 \le (c_1\,b_k)^p  \intl_{S^{n-1}} \rho_L^{kp}(\theta)\,\b_1(\theta)\, d_*\theta,\ee
where all parameters have the same meaning  as in (\ref{kiy}). In particular, for $p=n/k$ and $\mu=0$,
\be\label{mkkiyv3} \intl_{G_{n,k}}\, [V_k (L\cap \t_0)]^{n/k}\, \a_2 (\t_0)\, d_*\t_0\le c_2^{n/k}\, b_k^{n/k -1}\,
\vol_n  (L),\ee
where $\a_2$ and $c_2$ are the same as in (\ref{mkiyv3}).

We conclude this subsection by exhibiting a nice inequality for central sections of different dimensions. This inequality follows  from (\ref{00kiy}), (\ref{kiya2}), and (\ref{dkiya2p}),  if we set $\psi=\rho_L^j$. Specifically,
\be\label{000kiy} \intl_{G_{n,k}}\,[\tilde V_j (L\cap \t_0)]^p\, \a (\t_0)\, d_*\t_0 \le \left (\frac{c\,b_k}{b_j}\right )^p \intl_{G_{n,j}} [V_j (L\cap \z_0)]^p\,\b(\z_0)\, d_*\z_0, \ee
where all parameters have the same meaning as in Theorem \ref{mn1}. An interested reader may derive many  consequences of (\ref{000kiy}) by choosing different combinations of parameters.

\subsection{Unweighted Estimates} \label{unwe}

 Let us write  (\ref{lonm6}) for
$\vp=\rho_L^k$, where  $L$ is a measurable star set in $\rn$ of finite measure.  We obtain
\be\label{yyv31u}||F_k \rho_L^k ||_n \le ||\rho_L^k||_{n/k}\ee
or, by (\ref{kiya1}) and (\ref{kiya2}),
\be\label{mkkiyv31u} \intl_{G_{n,k}}\, [V_k (L\cap \t_0)]^{n}\,  d_*\t_0\le \frac{b_k^n}{b_n^k}\,[V_n  (L)]^k.\ee

The inequality (\ref{mkkiyv31u}) has an interesting history. If $k=n-1$ and $L$ is a convex body,  it was proved by Busemann \cite{Bu}, and is known as the {\it Busemann intersection inequality}. The case of convex bodies with any $0<k<n$  is due to Busemann and  Straus \cite{BS} and Grinberg \cite{Gr}; see also Gardner \cite [Corollary 9.4.5] {Ga06}.  The equality sign in (\ref{mkkiyv31u}) yields the celebrated Furstenberg-Tzkoni formula \cite{FT, M} for ellipsoids.  Further progress was made in Gardner's work  \cite{Ga07},   where  (\ref{mkkiyv31u}) was extended to arbitrary bounded Borel subsets of $\rn$.

Our approach to  (\ref{mkkiyv31u}), that relies on the corresponding inequality for the $k$-plane transform, shows that (\ref{mkkiyv31u}) actually holds for arbitrary (not necessarily bounded)  star set $L$ of finite measure.

For applications of (\ref{mkkiyv31u}), the reader is referred to \cite {CGL, DP, DPP, Ga06, Ga07, PP, PV, Z}. A similar inequality on the sphere and the  hyperbolic spaces was studied  by Dann,  Kim and  Yaskin \cite {DKY}.

 Following Lutwak's observation   \cite [p. 162,  (3)]{Lut93}, one can proceed in the opposite direction that may give an alternative  proof of the  Drury-Christ-Drouot  inequality (\ref{scamm}). Suppose that (\ref{mkkiyv31u}) has been  proved ``geometrically'' for star bodies $L$ with smooth boundary. Then, given a smooth nonnegative function $\vp$ on $\sn$, we can define a star set $L$ with the radial function  $\rho_L= \vp ^{1/k}$ and get $||F_k \vp ||_n \le ||\vp||_{n/k}$. The density argument extends this estimate  to all $\vp\in L^{n/k}(\sn)$, and the stereographic projection in Theorem \ref{lonm3p0} yields  (\ref{scamm}).

Further, for all $1\le j< k<n$, setting $\vp (\z_0)=(F_j \rho_L^k)(\z_0)=b_j^{-1}\tilde V_k (L\cap \z_0)$  in (\ref{alonm6}) and using the equality
\[ (F_{j,k}F_j \rho_L^k)(\t_0)= (F_k \rho_L^k)(\t_0)=b_k^{-1} V_k (L\cap \t_0),\]
 we  arrive at the following conjecture generalizing  (\ref{mkkiyv31u}).

 \begin{conjecture}\label {abji} For any  measurable star set $L$ in $\rn$ and any $1\le j< k<n$,
\be\label {ablonm6}
 \intl_{G_{n, k}} [V_k (L\cap \t_0)]^{n/j}\,  d_*\t_0  \le \left (\frac{b_k}{b_j}\right )^{n/j} \Big ( \intl_{G_{n, j}} [\tilde V_k (L\cap \z_0)]^{n/k}\, d_*\z_0 \Big )^{k/j}.\ee
More generally, choosing $\vp (\z_0)=(F_j \rho_L^m)(\z_0)=b_j^{-1}\tilde V_m (L\cap \z_0)$, we have
\be\label {ablonm6m}
 \intl_{G_{n, k}} [\tilde V_m (L\cap \t_0)]^{n/j}\,  d_*\t_0  \le \left (\frac{b_k}{b_j}\right )^{n/j} \Big ( \intl_{G_{n, j}} [\tilde V_m (L\cap \z_0)]^{n/k}\, d_*\z_0 \Big )^{k/j}.\ee
Here $m$ is an arbitrary real number for which the integral on the right-hand side  exists in the Lebesgue sense.
\end{conjecture}

If  (\ref{ablonm6}) and (\ref{ablonm6m}) are true, they are sharp because $\vp \equiv 1$ yields the equality sign in   (\ref{alonm6}).

\vskip 0.3 truecm

\noindent {\it Acknowledgement.} I am thankful to Artem Zvavitch who brought my attention to the preprint \cite{DKY}.

\bibliographystyle{amsplain}

\end{document}